\documentclass[12pt]{article}
\usepackage{amssymb}
\usepackage{amsfonts}
\usepackage{amsmath}
\usepackage{hyperref}
\usepackage{tikz}
\usepackage{CJK}
\usepackage{verbatim}

\usetikzlibrary{matrix,arrows}
\newtheorem{theorem}{Theorem}

\newtheorem{corollary}[theorem]{Corollary}

\newtheorem{lemma}[theorem]{Lemma}

\newtheorem{proposition}[theorem]{Proposition}
\newtheorem{remark}[theorem]{Remark}

\newenvironment{proof}[1][Proof]{\noindent\textbf{#1.} }{\ \rule{0.5em}{0.5em}}

\input{tcilatex}

\begin{document}

\title{Structures of Adjoint-Stable Algebras over Factorizable Hopf Algebras}
\author{Zhimin Liu \thanks{%
E-mail: zhiminliu13@fudan.edu.cn} \thanks{Project funded by China Postdoctoral Science Foundation grant 2019M661327}\hspace{1cm} Shenglin Zhu
\thanks{CONTACT: mazhusl@fudan.edu.cn} \\
Fudan University, Shanghai 200433, China}

\date{}
\maketitle

\begin{abstract}
  For a quasi-triangular Hopf algebra $\left( H,R\right) $, there is a
  notion of transmuted braided group $H_{R}$ of $H$ introduced by
  Majid. The transmuted braided group $H_{R}$ is a Hopf algebra in the
  braided category $ _{H}\mathcal{M}$. The $R$-adjoint-stable algebra
  associated with any simple left $H_{R}$-comodule is defined by the
  authors, and is used to characterize the structure of all
  irreducible Yetter-Drinfeld modules in ${}_{H}^{H} \mathcal{YD}$. In
  this note, we prove for a semisimple factorizable Hopf algebra
  $ \left( H,R\right) $ that any simple subcoalgebra of $H_R$ is
  $H$-stable and the $R$-adjoint-stable algebra for any simple left
  $H_R$-comodule is anti-isomorphic to $H$. As an application, we
  characterize all irreducible Yetter-Drinfeld modules.
\end{abstract}


\textbf{2000 MATHEMATICS SUBJECT CLASSIFICATION}: 16W30

\section{Introduction}

Let $\left( H,R\right) $ be a quasi-triangular Hopf algebra. In the braided
category $_{H}\mathcal{M}$ of finite dimensional left $H$-modules, a Hopf
algebra $H_{R}$, named `the transmuted braided group' of $H$, was
constructed by Majid~\cite{Majid1991Braided}, and it is proved by Zhu-Zhang~%
\cite{Zhu2015Braided} that ${}_{H}^{H}\mathcal{YD}\cong {}_{H}^{H_{R}}%
\mathcal{M}$. As an object of ${}_{H}^{H_{R}}\mathcal{M}$, each
Yetter-Drinfeld module $V\in {}_{H}^{H}\mathcal{YD}$ gives rise to a
subcoalgebra $D_{V}$ of $H_{R}$. If $k$ is a field and $\left( H,R\right)
=\left( kG,1\otimes 1\right) $ is the group algebra of a finite group $G$,
then the associated subcoalgebra $D_{V}$ of an irreducible Yetter-Drinfeld
module $V$ is the subcoalgebra linearly spanned by a conjugacy class $C$ of $%
G$, and $V$ can be characterized by a module over $kC\left( g\right) $,
where $g\in C$ and $C\left( g\right) $ is the centralizer of $g$ in $G$ (see~\cite{dijkgraaf1992quasi,gould1993quantum}.) For
a semisimple and cosemisimple quasi-triangular Hopf algebra $\left(
H,R\right) $, the authors~\cite{LiuZhu2019On} introduced the notion of $R$%
-adjoint-stable algebras and used it to characterized irreducible
Yetter-Drinfeld modules.

The notion of factorizable Hopf algebra was introduced by Reshetikhin and
Semenov-Tian-Shansky~\cite{reshetikhin1988quantum}. A quasi-triangular Hopf
algebra $\left( H,R\right) $ is called factorizable if the linear map $%
H^{\ast }\rightarrow H$, $p\mapsto \left\langle p,{R_{1}}^{2}{R_{2}}%
^{1}\right\rangle {R_{1}}^{1}{R_{2}}^{2}$ is bijective. When $\left(
H,R\right) $ is factorizable, Lyubashenko and Majid~\cite%
{Lyubashenko1994Braided,Majid1995foundations} show that the braided Hopf
algebra $H_{R}$ is isomorphic to its dual Hopf algebra (see also~\cite%
{Bulacu2004Factorizable} for quasi-Hopf case).

In this short paper, we prove for a semisimple factorizable
Hopf algebra $ \left( H,R\right) $ that any simple subcoalgebra of
$H_R$ is $H$-stable, and the $R$-adjoint-stable algebra for any simple
left $H_R$-comodule is anti-isomorphic to $H$. As an application, we
characterize all irreducible Yetter-Drinfeld modules.

\section{Preliminaries}

We first recall some preliminaries and fix some notations. Throughout this
paper, $k$ is a field, all (co)algebras are over $k$, and $\left( H,R\right)
$ is always a finite dimensional quasi-triangular Hopf algebra. For detailed
knowledge of coalgebras, Yetter-Drinfeld modules, and transmuted braided
groups, one can refer to~\cite%
{MR0252485,Yetter1990Quantum,Majid1991Braided,Zhu2015Braided,LiuZhu2019On}.

For self-containedness, we recall some basic notions here. The transmuted
braided group $H_{R}$ is a Hopf algebra in the braided tensor category $_{H}%
\mathcal{M}$. Explicitly, $H_{R}$ is the left $H$-module algebra $H$ with
the left adjoint action $\cdot_{ad}$. Its coproduct $\Delta _{R}$ and
antipode $S_{R}$ are
\begin{equation*}
\Delta _{R}(h)=h_{(1)}S\left( R^{2}\right) \otimes R^{1}\cdot _{ad}h_{(2)},\
S_{R}\left( h\right) =R^{2}S\left( R^{1}\cdot _{ad}h\right) \text{, where }%
h\in H\text{.}
\end{equation*}

Any module $V\in {}_{H}^{H}\mathcal{YD}$ is a left $H_{R}$-comodule via $%
\rho _{R}:V\rightarrow H_{R}\otimes V$ by%
\begin{equation}
\rho _{R}\left( v\right) =v_{\left\langle -1\right\rangle }S\left(
R^{2}\right) \otimes R^{1}v_{\left\langle 0\right\rangle },\text{ where }%
v\in V. \label{eq_rho_R}
\end{equation}%
This structure makes $V$ an object of $_{H}^{H_{R}}\mathcal{M}$. Let
\begin{equation*}
D_{V}=\limfunc{span}\left\{ v_{\left\langle -1\right\rangle }S\left(
R^{2}\right) \left\langle v^{\ast },R^{1}v_{\left\langle 0\right\rangle
}\right\rangle \mid v\in V,v^{\ast }\in V^{\ast }\right\} .
\end{equation*}%
Then $D_{V}$ is an $H$-stable subcoalgebra of $H_{R}$, which is also a
Yetter-Drinfeld submodule of $\left( H,\cdot _{ad},\Delta \right) \in
{}_{H}^{H}\mathcal{YD}$. We call it the subcoalgebra of $H_{R}$ associated
with $V$. It is proved in~\cite{LiuZhu2019On} that the set of
Yetter-Drinfeld submodules of $H$ coincides with the set of $H$-stable
subcoalgebras of $H_{R}$. If $V\in {}_{H}^{H}\mathcal{YD}$ is irreducible
then $D_{V}$ is a minimal $H$-stable subcoalgebra of $H_{R}$, as well as an
irreducible Yetter-Drinfeld submodule of $H$.

Let $W$ be a finite dimensional left $H_{R}$-comodule. Then $H\otimes W$ is
a natural object in ${}_{H}^{H_{R}}\mathcal{M}$ with the $H$-action and $%
H_{R}$-coaction given by
\begin{equation}
h^{\prime }\left( h\otimes w\right) =\left( h^{\prime }h\otimes w\right) ,%
\text{ }\rho \left( h\otimes w\right) =h_{\left( 1\right) }\cdot
_{ad}w_{^{\left\langle -1\right\rangle }}\otimes h_{\left( 2\right) }\otimes
w_{\left\langle 0\right\rangle },  \label{stru HotimesW}
\end{equation}%
where $h,h^{\prime }\in H$, $w\in W$. The object $H\otimes W$ was used in
~\cite{LiuZhu2019On} to characterize the structure of irreducible
Yetter-Drinfeld modules over $H$.

Let $D=D_{H\otimes W}$ be the subcoalgebra of $H_{R}$ associated with $%
H\otimes W$, then $H\otimes W\in {}_{H}^{D}\mathcal{M}$, and $\rho\left(W\right)\subseteq D\otimes W$. On $N_{W}=W^{\ast }\square _{D}\left( H\otimes W\right) $,
where $W^{\ast }$ is the canonical right $D$-comodule induced from the left $%
D$-coaction of $W$, there is a natural algebra structure via
\begin{equation*}
x\circ y=\sum_{l=1}^{n}\sum_{j=1}^{m}v_{l}^{\ast }\otimes g_{l}h_{j}\otimes
\left\langle w_{j}^{\ast },v_{l}\right\rangle w_{j},
\end{equation*}%
where $x=\sum_{j=1}^{m}w_{j}^{\ast }\otimes h_{j}\otimes w_{j}$, $%
y=\sum_{l=1}^{n}v_{l}^{\ast }\otimes g_{l}\otimes v_{l}$ are elements in $%
N_{W}$. The algebra $N_{W}$ is termed the $R$-adjoint-stable algebra of $W$.

Define a left $N_{W}$ module structure on $H\otimes W$ by $\left(
\sum_{j}w_{j}^{\ast }\otimes h_{j}\otimes w_{j}\right) \cdot \left( h\otimes
w\right) =\sum_{j}hh_{j}\otimes w_{j}\left\langle w_{j}^{\ast
},w\right\rangle $. For any right $N_{W}$-module $U$, $U\otimes
_{N_{W}}\left( H\otimes W\right) \in {}_{H}^{D}\mathcal{M}$ with the $H$%
-module structure and $D$-comodule induced by that on $H\otimes W$. Let $%
V\in {}_{H}^{D}\mathcal{M}$, $W^{\ast }\square _{D}V$ is a right $N_{W}$%
-module via
\begin{equation}
\left( \sum_{i}w_{i}^{\prime \ast }\otimes v_{i}\right) \cdot \left(
\sum_{j}w_{j}^{\ast }\otimes h_{j}\otimes w_{j}\right)
=\sum_{i}\sum_{j}w_{j}^{\ast }\otimes h_{j}v_{i}\left\langle w_{i}^{\prime
\ast },w_{j}\right\rangle ,  \label{W*_cpd_V_as_N_W_module}
\end{equation}%
for $\sum_{i}w_{i}^{\prime \ast }\otimes v_{i}\in W^{\ast }\square _{D}V$, $%
\sum_{j}w_{j}^{\ast }\otimes h_{j}\otimes w_{j}\in N_{W}$.

The following lemma is~\cite[Theorem 5.6]{LiuZhu2019On}, which we need later
on.

\begin{lemma}
\label{lemma LiuZhu}Let $\left( H,R\right) $ be a semisimple and
cosemisimple quasi-triangular Hopf algebra, and $W$ be a finite dimensional
left $H_{R}$-comodule. Write $D=D_{H\otimes W}$, then the functors
\begin{equation*}
W^{\ast }\square _{D}\bullet :{}_{H}^{D}\mathcal{M}\rightarrow \mathcal{M}%
_{N_{W}}\text{ and }\bullet \otimes _{N_{W}}\left( H\otimes W\right) :%
\mathcal{M}_{N_{W}}\rightarrow {}_{H}^{D}\mathcal{M}
\end{equation*}%
define a category equivalence.
\end{lemma}

\section{Adjoint-Stable Algebras for Factorizable Hopf Algebras}

Assume that $\left( H,R\right) $ is a finite dimensional quasi-triangular
Hopf algebra. Let $\left( H_{R}\right)^{\ast }$ denote the dual Hopf
algebra of $H_{R}$ in the category $_{H}\mathcal{M}$. Then $\left(
H_{R}\right) ^{\ast }=H^{\ast }$ as vector space, and as an object of $_{H}%
\mathcal{M}$ the left $H$-module structure on $\left( H_{R}\right) ^{\ast }$
is determined by $\left\langle h\rightharpoonup \!\!\!\!\rightharpoonup
f,h^{\prime }\right\rangle =\left\langle f,S\left( h\right) \cdot
_{ad}h^{\prime }\right\rangle $, for $h,h^{\prime }\in H,f\in H^{\ast }$.
The multiplication and comultiplication on $\left( H_{R}\right) ^{\ast }$
are defined by
\begin{eqnarray*}
f\ast _{R}g &=&\left( S\left( {R_{1}}^{2}{R_{2}}^{2}\right) \rightharpoonup
g\right) \ast \left( S\left( {R_{2}}^{1}\right) \rightharpoonup
f\leftharpoonup {R_{1}}^{1}\right) , \\
\Delta \left( f\right) &=&f_{\left( 2\right) }\otimes f_{\left( 1\right) },
\end{eqnarray*}%
for $f,g\in H^{\ast }$. One observes that the algebra $\left(H_{R}\right)^{\ast }$ here
is opposite to the usual convolution algebra of $H_{R}$.

Define a map
\begin{equation}
\Phi :\left( H_{R}\right) ^{\ast }\rightarrow H_{R},\ f\mapsto \left\langle
f,S\left( {R_{2}}^{2}{R_{1}}^{1}\right) \right\rangle {R_{2}}^{1}{R_{1}}^{2}.
\label{eq Phi}
\end{equation}

The following result is due to Lyubashenko and Majid~\cite%
{Lyubashenko1994Braided,Majid1995foundations}.

\begin{lemma}
\label{lemma H_R selfsual}The map $\Phi $ is a morphism of braided Hopf
algebras in $_{H}\mathcal{M}$. If $\left( H,R\right) $ is factorizable, then
$H_{R}$ is a self-dual braided Hopf algebra in $_{H}\mathcal{M}$.
\end{lemma}

Now we take a nonzero module $W\in {}_{H}\mathcal{M}$, then $\Phi $ induced
a left $H_{R}$-comodule structure on $W$ via
\begin{equation}
\rho _{R}\left( w\right) =S\left( {R_{2}}^{2}{R_{1}}^{1}\right) \otimes {%
R_{2}}^{1}{R_{1}}^{2}w,\ w\in W.  \label{rho_R W}
\end{equation}%
If $\left( H,R\right) $ is factorizable, then every left $H_{R}$-comodule $W$
is of this form by Lemma~\ref{lemma H_R selfsual}, thus in this case we can
identify $_{H}\mathcal{M}={}^{H_{R}}\mathcal{M}$. We first compute the
R-adjoint-stable algebra of $W$ with the given $H_{R}$-comodule structure.

Let $W,M$ be two left $H$-modules. We define a left $H$-action and a left $%
H_{R}$-coaction on $W\otimes M$ via
\begin{equation*}
h\left( w\otimes m\right) =h_{\left( 1\right) }w\otimes h_{\left( 2\right)
}m,\text{ }\rho _{R}\left( w\otimes m\right) =w_{\left\langle
-1\right\rangle }\otimes w_{\left\langle 0\right\rangle }\otimes m,\ h\in
H,w\in W,m\in M,
\end{equation*}%
where $w_{\left\langle -1\right\rangle }\otimes w_{\left\langle
0\right\rangle }=\rho _{R}\left( w\right) $ as indicated in (\ref{rho_R W}).
We denote $W\otimes M$ with these $H$-action and $H_{R}$-coaction by $%
\overline{W\otimes M}$.

\begin{lemma}
\label{lemma HotWisoWotH}Let $W,M$ be two left $H$-modules. Then

\begin{enumerate}
\item $\overline{W\otimes M}\in {}_{H}^{H_{R}}\mathcal{M}$. In particular,
the left $H$-module $W=\overline{W\otimes k}$ is a natural object of $%
{}_{H}^{H_{R}}\mathcal{M}$.

\item The map
\begin{equation}
H\otimes W\rightarrow \overline{W\otimes H},\text{ }h\otimes w\mapsto
h_{\left( 1\right) }w\otimes h_{\left( 2\right) },  \label{eqLemma2)}
\end{equation}
is an isomorphism in ${}_{H}^{H_{R}}\mathcal{M}$, where $H\otimes W\in
{}_{H}^{H_{R}}\mathcal{M}$ is given by (\ref{stru HotimesW}).

\item The subcoalgebra $D_{H\otimes W}$ equals to
\begin{equation*}
D_{W}=\limfunc{span}\left\{ w_{\left\langle -1\right\rangle }\left\langle
w^{\ast },w_{\left\langle 0\right\rangle }\right\rangle \mid w\in W,w^{\ast
}\in W^{\ast }\right\} ,
\end{equation*}%
and the map $\overline{W\otimes W^{\ast }}\rightarrow D_{W}$, sending $%
w\otimes w^{\ast }$ to $w_{\left\langle -1\right\rangle }\left\langle
w^{\ast },w_{\left\langle 0\right\rangle }\right\rangle $, is an epimorphism
in ${}_{H}^{H_{R}}\mathcal{M}$, where $W^{\ast }$ is a left $H$-module via $%
\left\langle hw^{\ast },w\right\rangle =\left\langle w^{\ast },S^{-1}\left(
h\right) w\right\rangle $.
\end{enumerate}
\end{lemma}

\begin{proof}
\begin{enumerate}
\item To see this, let $h\in H$, $w\in W$, $m\in M$, then%
\begin{eqnarray*}
\rho _{R}\left( h\left( w\otimes m\right) \right) &=&\rho _{R}\left(
h_{\left( 1\right) }w\otimes h_{\left( 2\right) }m\right) \\
&=&\left( h_{\left( 1\right) }w\right) _{\left\langle -1\right\rangle
}\otimes \left( h_{\left( 1\right) }w\right) _{\left\langle 0\right\rangle
}\otimes h_{\left( 2\right) }m \\
&=&S\left( {R_{2}}^{2}{R_{1}}^{1}\right) \otimes {R_{2}}^{1}{R_{1}}%
^{2}h_{\left( 1\right) }w\otimes h_{\left( 2\right) }m \\
&=&h_{\left( 1\right) }S\left( {R_{2}}^{2}{R_{1}}^{1}h_{\left( 2\right)
}\right) \otimes {R_{2}}^{1}{R_{1}}^{2}h_{\left( 3\right) }w\otimes
h_{\left( 4\right) }m \\
&=&h_{\left( 1\right) }\cdot _{ad}S\left( {R_{2}}^{2}{R_{1}}^{1}\right)
\otimes h_{\left( 2\right) }{R_{2}}^{1}{R_{1}}^{2}w\otimes h_{\left(
3\right) }m \\
&=&h_{\left( 1\right) }\cdot _{ad}w_{\left\langle -1\right\rangle }\otimes
h_{\left( 2\right) }\left( w_{\left\langle 0\right\rangle }\otimes m\right) .
\end{eqnarray*}

\item The inverse of (\ref{eqLemma2)}) is given by
\begin{equation*}
w\otimes h\mapsto S^{-1}\left( h_{\left( 1\right) }\right) w\otimes
h_{\left( 2\right) }.
\end{equation*}%
Using the fact that $W\in {}_{H}^{H_{R}}\mathcal{M}$, it is verified easily
that (\ref{eqLemma2)}) preserves the structures.

\item It follows directly from 2) and an easy verification.
\end{enumerate}
\end{proof}

\begin{remark}
  Observe that the statements of Lemma~\ref{lemma HotWisoWotH} is valid for $H_R$-comodule $W$ with its comodule structure
   arises from a left $H$-module. For a general 
  left $H_{R}$-comodule $W$, let $C$ be the subcoalgebra of $H_{R}$
  associated with $W$, then $D_{H\otimes W}=H\cdot _{ad}C$ is not necessarily
  equal to $C$.
\end{remark}

Let $W$ be a left $H$-module. Then $\left( \func{End}^{H_{R}}\left( W\right)
\right) ^{op}$ is a left $H$-module algebra via
\begin{equation*}
\left( h\cdot \alpha \right) \left( w\right) =h_{\left( 2\right) }\alpha
\left( S^{-1}\left( h_{\left( 1\right) }\right) w\right) ,
\end{equation*}%
for $h\in H$, $\alpha \in \func{End}^{H_{R}}\left( W\right) $, $w\in W$.

\begin{proposition}
\label{prop stru of N_W}Let $W$ be a finite dimensional left $H$-module.
Then the $R$-adjoint stable algebra $N_{W}$ is anti-isomorphic to $\left(
\func{End}^{H_{R}}\left( W\right) \right) ^{op}\#H$, where $D=D_{W}$.
\end{proposition}

\begin{proof}
We write $\left( \func{End}^{H_{R}}\left( W\right) \right) ^{op}$ as $%
W^{\ast }\square _{D}W$ with multiplication
\begin{equation*}
\left( \sum_{i}w_{i}^{\ast }\otimes w_{i}\right) \left( \sum_{j}v_{j}^{\ast
}\otimes v_{j}\right) =\sum_{i}\sum_{j}\left\langle v_{j}^{\ast
},w_{i}\right\rangle w_{i}^{\ast }\otimes v_{j},
\end{equation*}%
for $\sum_{i}w_{i}^{\ast }\otimes w_{i},\ \sum_{j}v_{j}^{\ast }\otimes
v_{j}\in W^{\ast }\square _{D}W$.

By 2) of Lemma~\ref{lemma HotWisoWotH} and the definition of $N_W$, we have
a linear isomorphism
\begin{eqnarray*}
\theta&:& N_W=W^*\square_D\left(H\otimes W\right)\to W^*\square_D\left(%
\overline{W\otimes H}\right)\to \left( W^{\ast }\square _{D}W\right) \# H, \\
&& \sum_{i}w_{i}^{\ast }\otimes h_{i}\otimes w_{i}\mapsto
\sum_{i}\left(w_{i}^{\ast }\otimes h_{i\left( 1\right) }w_{i}\right)\#
h_{i\left( 2\right)}.
\end{eqnarray*}
We will show that $\theta $ is an anti-algebra isomorphism. Let $%
x=\sum_{i}w_{i}^{\ast }\otimes h_{i}\otimes w_{i}$,\newline
$y=\sum_{j}v_{j}^{\ast }\otimes g_{j}\otimes v_{j}\in N_{W}$, then
\begin{eqnarray*}
\theta \left( y\right) \theta \left( x\right) &=&\left(
\sum_{j}\left(v_{j}^{\ast }\otimes g_{j\left( 1\right) }v_{j}\right)\#
g_{j\left( 2\right) }\right) \left( \sum_{i}\left(w_{i}^{\ast }\otimes
h_{i\left( 1\right) }w_{i}\right)\# h_{i\left( 2\right) }\right) \\
&=&\sum_{j}\sum_{i}\left( v_{j}^{\ast }\otimes g_{j\left( 1\right)
}v_{j}\right) \left( g_{j\left( 2\right) }\cdot \left( w_{i}^{\ast }\otimes
h_{i\left( 1\right) }w_{i}\right) \right) \# g_{j\left( 3\right) }h_{i\left(
2\right) } \\
&=&\sum_{j}\sum_{i}\left( v_{j}^{\ast }\otimes \left( g_{j\left( 2\right)
}\cdot \left( w_{i}^{\ast }\otimes h_{i\left( 1\right) }w_{i}\right) \right)
\left( g_{j\left( 1\right) }v_{j}\right) \right) \# g_{j\left( 3\right)
}h_{i\left( 2\right) } \\
&=&\sum_{j}\sum_{i}\left( v_{j}^{\ast }\otimes \left\langle w_{i}^{\ast
},S^{-1}\left( g_{j\left( 2\right) }\right) \left( g_{j\left( 1\right)
}v_{j}\right) \right\rangle g_{j\left( 3\right) }h_{i\left( 1\right)
}w_{i}\right) \# g_{j\left( 3\right) }h_{i\left( 2\right) } \\
&=&\sum_{j}\sum_{i}\left(\left\langle w_{i}^{\ast },v_{j}\right\rangle
v_{j}^{\ast }\otimes g_{j\left( 1\right) }h_{i\left( 1\right)
}w_{i}\right)\# g_{j\left( 2\right) }h_{i\left( 2\right) } \\
&=&\theta \left( \sum_{i} \sum_{j}\left\langle w_{i}^{\ast
},v_{j}\right\rangle v_{j}^{\ast }\otimes g_{j}h_{i}\otimes w_{i}\right) \\
&=&\theta \left( xy\right) ,
\end{eqnarray*}%
as desired.
\end{proof}

From now on, we assume that $\left( H,R\right) $ is a semisimple
factorizable Hopf algebra over an algebraically closed field $k$. Then as a $%
k$-algebra, $H_{R}=H$ is semisimple. So by Lemma~\ref{lemma H_R selfsual}, $%
\left( H_{R}\right) ^{\ast }$ is semisimple, and then $H_{R}$ is
cosemisimple. Thus, the categories ${}_{H}^{H}\mathcal{YD\cong {}}%
_{H}^{H_{R}}\mathcal{M}$ are semisimple. Let $\limfunc{Irr}\left( H\right) $
denote a set of representatives of isomorphism classes of irreducible left $%
H $-modules.

As an immediate consequence of Proposition~\ref{prop stru of N_W}, we obtain:

\begin{corollary}
\label{cor Nw iso H}Let $W\in \limfunc{Irr}\left( H\right) $, then the $R$%
-adjoint stable algebra $N_{W}$ is isomorphic to $H^{op}$.
\end{corollary}

\begin{proof}
$W$ is simple in $^{H_{R}}\mathcal{M}$, since the map $\Phi :H^{\ast
}\rightarrow H$ given by (\ref{eq Phi}) is bijective. It follows that $\func{%
End}^{H_{R}}\left( W\right) \cong k$. So by Proposition~\ref{prop stru of
N_W},
\begin{equation*}
N_{W}\cong \left( k\#H\right) ^{op}\cong H^{op}.
\end{equation*}
\end{proof}

\begin{remark}
\label{Remark iso bt H and Nw}Let $\left\{ w_{i},w_{i}^{\ast }\mid
i=1,\ldots ,n\right\} $ be a dual basis for $W$. Then the inverse of $\theta$
is the map $\psi :H^{op}\rightarrow N_{W}$ given by
\begin{equation*}
\psi \left( h\right) =\sum_{i=1}^{n}w_{i}^{\ast }\otimes h_{\left( 2\right)
}\otimes S^{-1}\left( h_{\left( 1\right) }\right) w_{i}.
\end{equation*}
\end{remark}

\begin{lemma}
\label{Lemma Hstsblesubcoalgebra}Every subcoalgebra of $H_{R}$ is $H$-stable.
\end{lemma}

\begin{proof}
By Lemma~\ref{lemma H_R selfsual}, $H_{R}\cong \left( H_{R}\right) ^{\ast }$
as braided Hopf algebras. So it suffices to show that every subcoalgebra of $%
\left( H_{R}\right) ^{\ast }$ is $H$-stable. Let $C$ be a subcoalgebra of $%
\left( H_{R}\right) ^{\ast }$. For any $h\in H$, $f\in C$,
\begin{equation*}
h\rightharpoonup \!\!\!\! \rightharpoonup f=S^{2}\left( h_{\left( 2\right) }\right)
\rightharpoonup f\leftharpoonup S\left( h_{\left( 1\right) }\right) \in C,
\end{equation*}%
since $\Delta \left( C\right) \subseteq C\otimes C$. Hence, $C$ is $H$%
-stable.
\end{proof}

\begin{proposition}
\label{Proposition D_W is simple} Let $T$ be the set of simple subcoalgebras
of $H_{R}$, then the map $\func{Irr}\left( H\right) \rightarrow T,\ W\mapsto
D_{W}$ is a bijection.

Moreover, $H=\tbigoplus_{W\in \func{Irr}\left( H\right) }D_{W}$ is a direct
sum of irreducible Yetter-Drinfeld modules.
\end{proposition}

\begin{proof}
It is known that if $W\in \func{Irr}\left( H\right) $, then $W$ is also a
simple left $H_{R}$-comodule, so the subcoalgebra $D_{W}$ associated with $W$
is simple. Since $_{H}\mathcal{M}={}^{H_{R}}\mathcal{M}$, $\left\{ D_{W}\mid
W\in \limfunc{Irr}\left( H\right) \right\} =T$.

As a cosemisimple coalgebra, $H_{R}=\tbigoplus_{D\in T }D= \tbigoplus_{W\in
\limfunc{Irr}\left( H\right) }D_{W}$. Since each $H$-stable subcoalgebra of $%
H_R$ corresponds to a Yetter-Drinfeld submodule of $H$, the result follows
from Lemma~\ref{Lemma Hstsblesubcoalgebra}.
\end{proof}

\begin{remark}
For a general quasi-triangular Hopf algebra $\left( H,R\right) $, all
irreducible Yetter-Drinfeld submodules of $H\in {}_{H}^{H}\mathcal{YD}$ are
subcoalgebras of $H_{R}$ \cite[Proposition 3.5]{LiuZhu2019On}, but they are
not always simple even if $H$ is semisimple and cosemisimple.
Counterexamples are the group algebra of a finite nonabelian group and the
Kac-Paljukin 8-dimensional Hopf algebra~\cite[Example 5.16]{LiuZhu2019On}.
\end{remark}

As a corollary, we conclude the following result due to Schneider.

\begin{corollary}[{\protect\cite[Theorem 3.2]{schneider2001some}}]
Let $\left( H,R\right) $ be a semisimple factorizable Hopf algebra over $k$.
If $W$ is a simple left $H$-module, then $\left( \dim W\right) ^{2}$ divides
$\dim H$.
\end{corollary}

\begin{proof}
The coalgebra $D_{W}$ associated with the Yetter-Drinfeld module $W$ is
simple by Proposition~\ref{Proposition D_W is simple}, so $\left( \dim
W\right) ^{2}=\dim D_{W}$. As $D_{W}$ is an irreducible Yetter-Drinfeld
submodule of $H\in {}_{H}^{H}\mathcal{YD}$, $\dim D_{W}\mid \dim H$ by a
well-known result of Etingof and Gelaki~\cite[Theorem 1.4]{Etingof1997Some}.
\end{proof}

Finally, we characterize all irreducible Yetter-Drinfeld modules in $%
{}_{H}^{H}\mathcal{YD}$.

\begin{theorem}
Let $\left( H,R\right) $ be a semisimple factorizable Hopf algebra over an
algebraically closed field $k$. Then for any simple left $H$-modules $W$ and
$M$, $\overline{W\otimes M}$ is an irreducible object of ${}_{H}^{H}\mathcal{%
YD}$. Conversely, every irreducible Yetter-Drinfeld module in ${}_{H}^{H}%
\mathcal{YD}$ is isomorphic to $\overline{W\otimes M}$ for some $\left(
W,M\right) \in \limfunc{Irr}\left( H\right) \times \limfunc{Irr}\left(
H\right) $.

Furthermore, any irreducible Yetter-Drinfeld submodule $D$ of $H$ is of the
form $\overline{W\otimes W^{\ast }}$ for some $W\in \func{Irr}H$, and in
this case, $D_{W}=D$.
\end{theorem}

\begin{proof}
Let $W,M\in \limfunc{Irr}\left( H\right) $. Then $W^{\ast }\square
_{D}\left( \overline{W\otimes M}\right) $ has a natural right $N_{W}$-module
structure via (\ref{W*_cpd_V_as_N_W_module}). By Corollary~\ref{cor Nw iso H}%
, $N_{W}\cong H^{op}$ as algebras, and this isomorphism is given by $\psi $
defined as in Remark~\ref{Remark iso bt H and Nw}. So $W^{\ast }\square
_{D}\left( \overline{W\otimes M}\right) $ is a left $H$-module via $\psi $.
Let $\left\{ w_{i},w_{i}^{\ast }\mid i=1,\ldots ,n\right\} $ be a dual basis
for $W$. Since $W^{\ast }\square _{D}\left( \overline{W\otimes M}\right)
\cong \func{End}^{D}\left( W\right) \otimes M\cong k\otimes M=M$, we may
write elements of $W^{\ast }\square _{D}\left( \overline{W\otimes M}\right) $
as $\sum_{i=1}^{n}w_{i}^{\ast }\otimes w_{i}\otimes m$ with $m\in M$. For
any $h\in H$, $m\in M$,
\begin{eqnarray*}
h\cdot \left( \sum_{i=1}^{n}w_{i}^{\ast }\otimes w_{i}\otimes m\right)
&=&\left( \sum_{i=1}^{n}w_{i}^{\ast }\otimes w_{i}\otimes m\right) \cdot
\psi \left( h\right) \\
&=&\left( \sum_{i=1}^{n}w_{i}^{\ast }\otimes w_{i}\otimes m\right) \cdot
\left( \sum_{j=1}^{n}w_{j}^{\ast }\otimes h_{\left( 2\right) }\otimes
S^{-1}\left( h_{\left( 1\right) }\right) w_{j}\right) \\
&=&\sum_{i=1}^{n}\sum_{j=1}^{n}\left\langle w_{i}^{\ast },S^{-1}\left(
h_{\left( 1\right) }\right) w_{j}\right\rangle w_{j}^{\ast }\otimes
h_{\left( 2\right) }\left( w_{i}\otimes m\right) \\
&=&\sum_{j=1}^{n}w_{j}^{\ast }\otimes h_{\left( 2\right) }S^{-1}\left(
h_{\left( 1\right) }\right) w_{j}\otimes h_{\left( 3\right) }m \\
&=&\sum_{j=1}^{n}w_{j}^{\ast }\otimes w_{j}\otimes hm.
\end{eqnarray*}%
Hence, $W^{\ast }\square _{D}\left( \overline{W\otimes M}\right) \cong M$ as
a right $N_{W}$-module. Now apply Lemma~\ref{lemma LiuZhu} to see that
\begin{equation*}
\overline{W\otimes M}\cong \left( W^{\ast }\square _{D}\left( \overline{%
W\otimes M}\right) \right) \otimes _{N_{W}}\left( H\otimes W\right) \cong
M\otimes _{N_{W}}\left( H\otimes W\right) .
\end{equation*}%
Since $M$ is a simple $N_W$-module, it follows that $\overline{W\otimes M}$ is
a simple object of $_{H}^{D_{W}}\mathcal{M}$. Moreover, the set $\left\{
\overline{W\otimes M}\mid M\in \limfunc{Irr}\left( H\right) \right\} $ forms
a complete set of representatives of the isomorphism classes of simple
objects in $_{H}^{D_{W}}\mathcal{M}$.

By Proposition~\ref{Proposition D_W is simple}, $H_{R}=\tbigoplus_{W\in
\limfunc{Irr}\left( H\right) }D_{W}$, so the category ${}_{H}^{H}\mathcal{YD}%
={}_{H}^{H_{R}}\mathcal{M=}\tbigoplus_{W\in \limfunc{Irr}\left( H\right)
}{}_{H}^{D_{W}}\mathcal{M}$. Thus, every simple object of ${}_{H}^{H}%
\mathcal{YD}$ is determined up to isomorphism by a pair\\ $\left( W,M\right)
\in \limfunc{Irr}\left( H\right) \times \limfunc{Irr}\left( H\right) $.

Assume that $D$ is an irreducible Yetter-Drinfeld submodule of $H$. Again by
Proposition~\ref{Proposition D_W is simple} $D=D_{W}$ for some $W\in
\limfunc{Irr}\left( H\right) $. It follows from Lemma~\ref{lemma HotWisoWotH}
that there exists a surjection $\overline{W\otimes W^{\ast }}\rightarrow
D_{W}$ of Yetter-Drinfeld modules. So $\overline{W\otimes W^{\ast }}\cong
D_{W}$, since $\overline{W\otimes W^{\ast }}$ is simple~in ${}_{H}^{H}%
\mathcal{YD}$.
\end{proof}

For a factorizable Hopf algebra $H$, a Hopf algebra isomorphism between the
Drinfeld double $D\left( H\right) $ and a twist of the usual tensor product
Hopf algebra $H\otimes H$ was given by Schneider (\cite[Theorem~4.3]%
{schneider2001some}). Using the result of Schneider, one can also describe
a~Yetter-Drinfeld~module over $H$ in terms of two $H$-modules.


\begin{thebibliography}{10}

\bibitem{Bulacu2004Factorizable}
D.~Bulacu and B.~Torrecillas.
\newblock Factorizable quasi-{Hopf} algebras--applications.
\newblock {\em J. Pure Appl. Algebr.}, 194(1-2):39--84, 2004.

\bibitem{dijkgraaf1992quasi}
R.~Dijkgraaf, V.~Pasquier, and P.~Roche.
\newblock Quasi {H}opf algebras, group cohomology and orbifold models.
\newblock {\em Nuclear Phys. B Proc. Suppl.}, 18B:60--72 (1991), 1990.
\newblock Recent advances in field theory (Annecy-le-Vieux, 1990).

\bibitem{Etingof1997Some}
P.~Etingof and S.~Gelaki.
\newblock Some properties of finite-dimensional semisimple {H}opf algebras.
\newblock {\em Math. Res. Lett.}, 5(1-2):191--197, 1998.

\bibitem{gould1993quantum}
M.~D. Gould.
\newblock Quantum double finite group algebras and their representations.
\newblock {\em Bull. Austral. Math. Soc.}, 48(2):275--301, 1993.

\bibitem{LiuZhu2019On}
Z.~Liu and S.~Zhu.
\newblock On the structure of irreducible {Y}etter-{D}rinfeld modules over
  quasi-triangular {H}opf algebras.
\newblock {\em J. Algebra}, 539:339--365, 2019.

\bibitem{Lyubashenko1994Braided}
V.~Lyubashenko and S.~Majid.
\newblock Braided groups and quantum {F}ourier transform.
\newblock {\em J. Algebra}, 166(3):506--528, 1994.

\bibitem{Majid1991Braided}
S.~Majid.
\newblock Braided groups and algebraic quantum field theories.
\newblock {\em Lett. Math. Phys.}, 22(3):167--175, 1991.

\bibitem{Majid1995foundations}
S.~Majid.
\newblock {\em Foundations of quantum group theory}.
\newblock Cambridge University Press, Cambridge, 1995.

\bibitem{reshetikhin1988quantum}
N.~Y. Reshetikhin and M.~A. Semenov-Tian-Shansky.
\newblock Quantum {$R$}-matrices and factorization problems.
\newblock {\em J. Geom. Phys.}, 5(4):533--550 (1989), 1988.

\bibitem{schneider2001some}
H.-J. Schneider.
\newblock Some properties of factorizable {H}opf algebras.
\newblock {\em Proc. Amer. Math. Soc.}, 129(7):1891--1898, 2001.

\bibitem{MR0252485}
M.~E. Sweedler.
\newblock {\em {Hopf} algebras}.
\newblock Mathematics Lecture Note Series. W. A. Benjamin, Inc., New York,
  1969.

\bibitem{Yetter1990Quantum}
D.~N. Yetter.
\newblock Quantum groups and representations of monoidal categories.
\newblock {\em Math. Proc. Cambridge Philos. Soc.}, 108(2):261--290, 1990.

\bibitem{Zhu2015Braided}
H.~Zhu and Y.~Zhang.
\newblock Braided autoequivalences and quantum commutative bi-{G}alois objects.
\newblock {\em J. Pure Appl. Algebra}, 219(9):4144--4167, 2015.

\end{thebibliography}

\end{document}